\documentclass[10pt]{amsart}
\usepackage{geometry}                
\usepackage{graphicx}
\usepackage{amssymb}
\usepackage{epstopdf}
\usepackage[colorlinks]{hyperref}
\newtheorem{theorem}{Theorem}[section]
\newtheorem{corollary}[theorem]{Corollary}
\newtheorem{lemma}[theorem]{Lemma}
\newtheorem{proposition}[theorem]{Proposition}
\theoremstyle{definition}
\newtheorem{definition}[theorem]{Definition}
\newtheorem{remark}[theorem]{Remark}
\newtheorem{example}[theorem]{Example}
\numberwithin{equation}{section}
\usepackage[colorlinks]{hyperref}
\usepackage[lite]{amsrefs}
\usepackage{mathpazo}
\title{Some remarks on the chord index}
\author{Zhiyun Cheng}
\author{Hongzhu Gao}
\author{Mengjian Xu}
\address{School of Mathematical Sciences, Laboratory of Mathematics and Complex Systems, Beijing Normal University, Beijing 100875, China}
\email{czy@bnu.edu.cn}
\email{hzgao@bnu.edu.cn}
\email{xmjmath@mail.bnu.edu.cn}
\begin{document}
\begin{abstract}
In this paper we discuss how to define a chord index via smoothing a real crossing point of a virtual knot diagram. Several polynomial invariants of virtual knots and links can be recovered from this general construction. We also explain how to extend this construction from virtual knots to flat virtual knots.
\end{abstract}
\maketitle
\section{Introduction}
Suppose we are given a knot diagram $K$, it is easy to count the number of crossing points of $K$. However, this integer $c(K)$ is not a knot invariant due to the first and second Reidemeister moves. Fixing an orientation of $K$, then each crossing point contributes $+1$ or $-1$ to the writhe of $K$. Now the writhe, denoted by $w(K)$, is preserved by the second Reidemeister move but not the first Reidemeister move. As a natural idea, if we can assign an index (or a weight) to each crossing point which satisfies some required conditions, then it is possible that the weighted sum of crossing points gives rise to a knot invariant. One example of this kind of knot invariant is the quandle cocycle invariant introduced by J. S. Carter et al. in \cite{Car2003}. Roughly speaking, let $X$ be a finite quandle and $A$ an abelian group. For a fixed coloring $\mathcal{C}$ of $K$ by $Q$ and a given 2-cocycle $\phi\in\mathbf{Z}^2(X;A)$ one can associate a (Boltzmann) weight $B(\tau;\mathcal{C})$ to each crossing point $\tau$. This (Boltzmann) weight $B(\tau;\mathcal{C})$ has the following properties:
\begin{enumerate}
\item The weight of the crossing point involved in the first Reidemeister move is the identity element;
\item The weights of the two crossing points involved in the second Reidemeister move are equivalent;
\item The product of the weights of the three crossing points involved in the third Reidemeister move is invariant under the third Reidemeister move.
\end{enumerate}

Now let us extend the scope of our discussion from classical knots to virtual knots. The precise definition of virtual knots can be found in Section \ref{section2}. For virtual knots, motivated by the properties above we also want to find a suitable definition for the (chord) index. The following chord index axioms were proposed by the first author in \cite{Che2016}, which can be regarded as a generalization of the parity axioms introduced by Manturov in \cite{Man2010}.

\begin{definition}\label{axiom}
Assume for each real crossing point $c$ of a virtual link diagram, according to some rules we can assign an index (e.g. an integer, a polynomial, a group etc.) to it. We say this index satisfies the \emph{chord index axioms} if it satisfies the following requirements:
\begin{enumerate}
\item The real crossing point involved in $\Omega_1$ has a fixed index (with respect to a fixed virtual link);
\item The two real crossing points involved in $\Omega_2$ have the same indices;
\item The indices of the three real crossing points involved in $\Omega_3$ are preserved under $\Omega_3$ respectively;
\item The index of the real crossing point involved in $\Omega_3^v$ is preserved under $\Omega_3^v$;
\item The index of any real crossing point not involved in a generalized Reidemeister move is preserved under this move.
\end{enumerate}
\end{definition}

One can easily find that our third condition is sharper than the 2-cocycle condition in quandle cohomology theory. Thus, some indices satisfying the our chord index axioms can be regarded as special cases of the quandle 2-cocycle. Actually, in \cite{Che2016} we defined a chord index for each real crossing point of a virtual link diagram with the help of a given finite biquandle. This construction extends some known invariants (for example the writhe polynomial defined in \cite{Che2013}) from virtual knots to virtual links. It is still unknown whether this construction can be used to define some nontrivial indices on classical knot diagrams.

However, the definition of the (Boltzmann) weight mentioned above is not intrinsic. Actually, it depends on the choices of a finite quandle $Q$ and a 2-cocycle $\phi$. Similarly, the chord index discussed in \cite{Che2016} is not intrinsic. It depends on the choice of a finite biquandle.

The main aim of this paper is trying to introduce some intrinsic chord indices which satisfy the chord index axioms defined in Definition \ref{axiom}. Roughly speaking, for each real crossing point $c$ of a virtual knot diagram $K$, we associate a flat virtual knot/link (see Section \ref{section2} for the definition) to it. This flat virtual knot/link is simply obtained from $K$ by smoothing the crossing point $c$. It turns out that this flat virtual knot/link satisfies the chord index axioms. Let $\mathcal{M}_1^u$ be the free $\mathbf{Z}$-module generated by the set of all unoriented flat virtual knots. In other words, any element of $\mathcal{M}_1^u$ has the form $\sum\limits_in_i\widetilde{K}_i$, where $\widetilde{K}_i$ is a unoriented flat virtual knot and $n_i\neq0$ for only a finite number of $n_i\in\mathbf{Z}$. Similarly, we use $\mathcal{M}_2$ to denote the free $\mathbf{Z}$-module generated by the set of all oriented 2-component flat virtual links. Then by using the flat virtual knot/link-valued chord index we introduce two new virtual knot invariants, which take values in $\mathcal{M}_1^u$ and $\mathcal{M}_2$ respectively. The readers are recommended to compare these invariants with some other flat virtual knots/graphs-valued virtual knot invariants. For example, the polynomial $\nabla(K)$ \cite{Tur2004} which takes values in the polynomial algebra generated by nontrivial flat virtual knots, or the $sl(3)$ invariant introduced in \cite{Kau2014}, which is valued in a module whose generators are graphs.

Many ideas of this paper have their predecessors in the literature. We will explain how to recover some known invariants from our two invariants in detail.

It seems a little odd that a virtual knot invariant takes values in a module generated by flat virtual knots/links, since flat virtual knots/links themselves are not easy to distinguish. To this end, we extend our idea to flat virtual knots and define two flat virtual knot invariants taking values in $\mathcal{M}_1$ and $\mathcal{M}_2$ respectively. Here $\mathcal{M}_1$ can be regarded as the oriented version of $\mathcal{M}_1^u$, the reader is referred to Section \ref{section5} for the definition. The key point is that each nontrivial flat virtual knot will be mapped to a linear combination of some flat virtual knots/links which are ``simpler" than the flat virtual knot under consideration. Several properties of these invariants will be discussed in Section \ref{section5}.

Throughout this paper we will often abuse our notation, letting $K$ refer to a (virtual) knot diagram and the (virtual) knot itself. Let us begin our journey with a quick review of virtual knots and flat virtual knots.

\section{A quick review of virtual knots and flat virtual knots}\label{section2}
Virtual knot theory was introduced by L. Kauffman in \cite{Kau1999}, which can be regarded as an extension of the classical knot theory. Assume we are given an immersed circle on the plane which has finitely many transversal intersections. By replacing each intersection point with an overcrossing, a undercrossing or a virtual crossing we will obtain a virtual knot diagram. We say a pair of virtual knot diagrams are equivalent if they can be connected by a sequence of generalized Reidemeister moves, see Figure \ref{figure1}. A virtual knot can be considered as an equivalence class of virtual knot diagrams that are equivalent under generalized Reidemeister moves.
\begin{figure}
\centering
\includegraphics{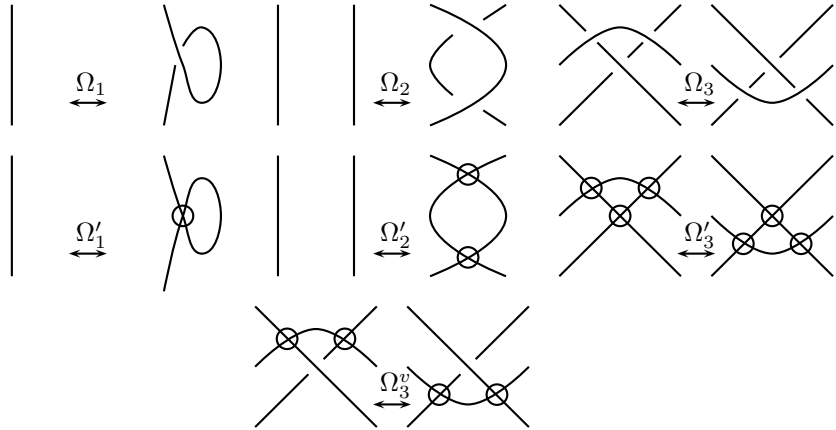}\\
\caption{Generalized Reidemeister moves}\label{figure1}
\end{figure}

Roughly speaking, there are two motivations to extend the classical knot theory to virtual knot theory. From the topological viewpoint, a classical knot is an embedding of $S^1$ into $R^3$ up to isotopies. It is equivalent to replace the ambient space $R^3$ with $S^2\times [0, 1]$. A virtual knot is an embedded circle in the thickened closed orientable surface $\Sigma_g\times [0, 1]$ up to isotopies and (de)stabilizations \cite{Car2002, Kup2003}. When $g=0$, a virtual knot is nothing but a circle in $S^2\times [0, 1]$, which is equivalent to a classical knot in $R^3$.

Besides of the topological interpretation, another motivation of introducing virtual knots is to find a planar realization for an arbitrary Gauss diagram (also called chord diagram). For a given classical knot diagram, it is well known that there exists a unique Gauss diagram corresponding to it. Actually, one can draw an embedded circle on the plane which represents the preimage of the knot diagram. Then for each crossing point we connect the two preimages with an arrow, directed from the preimage of the overcrossing to the preimage of the undercrossing. Finally we label a sign on each arrow according to the writhe of the corresponding crossing point. However, it is easy to observe that there exist some Gauss diagrams which cannot be realized as a classical knot diagram. Therefore one has to add some virtual crossing points. Although for a fixed Gauss diagram there are infinitely many virtual knot diagrams corresponding to it, it was proved in \cite{Kau1999, Gou2000} that all these virtual knot diagrams are equivalent. Namely, a Gauss diagram uniquely defines a virtual knot. See Figure \ref{figure2} for an example of the virtual trefoil knot and the corresponding Gauss diagram. Throughout this paper we will use the same symbol to denote a real crossing point and the corresponding chord.
\begin{figure}[h]
\centering
\includegraphics{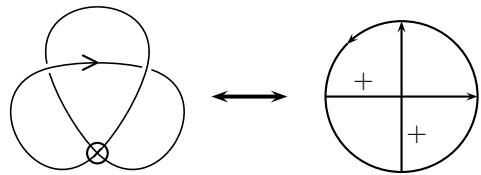}\\
\caption{Virtual trefoil knot and its Gauss diagram}\label{figure2}
\end{figure}

The importance of the realizations of all Gauss diagrams stems from the Gauss diagram representations of finite type invariants. It was first noticed by Polyak and Viro \cite{Pol1994} that some finite type invariants of degree 2 and 3 can be expressed in terms of Gauss diagrams. Later this result was extended to finite type invariants of all degrees in \cite{Gou2000}. More precisely, let $\mathcal{A}=\mathbf{Z}[D]$, where $D\in\mathcal{D}$, the set of all Gauss diagrams. Then the \emph{Polyak algebra} $\mathcal{P}$ is defined to be $\mathcal{A}/\mathcal{R}$, where $\mathcal{R}$ is generated by elements of the form $D-D'$ if $D$ and $D'$ are related by a Reidemeister move. By putting $A=0$ for any $A\in\mathcal{P}$ with more than $n$ chords we obtain the \emph{truncated Polyak algebra} $\mathcal{P}_n$. For two Gauss diagrams $D_1$ and $D_2$, consider the pairing
\begin{center}
$<D_1, D_2>=
\begin{cases}
1& \text{if }D_1=D_2\\
0& \text{if }D_1\neq D_2
\end{cases}$
\end{center}
and extend it linearly one obtains an inner product $< , >:\mathcal{A}\times \mathcal{A}\rightarrow\mathbf{Z}$. Then the main result of \cite{Gou2000} tells us that any integer-valued finite type invariant of degree $n$ of a (long) knot $D$ can be represented as $<A, \sum\limits_{\text{subdiagram }D'\subset D}D'>$ for some $A\in\mathcal{P}_n$. Note that in the world of classical knots, not every subdiagram of $D$ makes sense. For this reason, it is more reasonable to study the finite type invariants in the world of virtual knots.

Before continuing, it would be good to revisit the chord index from the viewpoint of finite type invariant. It is well known that there is no finite type invariant of degree 1 in classical knot theory. It is quite easy to observe this from the result above. Actually, the Gauss diagram with only one chord is unique, hence the inner product above just gives the crossing number of knot diagram. Even if using the signed Gauss diagram with one chord, one obtains the writhe of the knot diagram. Neither of them is a knot invariant. The benefit of the chord index is being able to define a finite type invariant of degree 1 by assigning an index to each chord. This explains the reason why we can define some finite type invariants of degree 1 for virtual knots, see \cite{Saw2003} and \cite{Hen2010} for some concrete examples. Furthermore, if we can find some nontrivial chord indices for classical knots, then it is possible that the weighted sum of all crossing points defines a knot invariant, even if the (signed) sum of all crossing points does not.

Virtual knots not only have finite type invariants of degree 1, but also finite type invariants of degree 0. This reveals the the non-triviality of flat virtual knots. Roughly, flat virtual knots (which was named as virtual strings in \cite{Tur2004}) can be regarded as virtual knots without over/undercrossing informations. More precisely, a flat virtual knot diagram can be obtained from a virtual knot diagram by replacing all real crossing points with flat crossing points. In other words, there are only two kinds of crossing points in a flat virtual knot diagram, flat and virtual. Similarly, by replacing all real crossing points with flat crossing points in Figure \ref{figure1} one obtains the flat generalized Reidemeister moves. We will use $F\Omega_1, F\Omega_2, F\Omega_3, F\Omega_1', F\Omega_2', F\Omega_3'$ and $F\Omega^v_3$ to denote them. Then flat virtual knots can be defined as the equivalence classes of flat virtual knot diagrams up to flat generalized Reidemeister moves. It is evident to notice that if a flat virtual knot diagram has only flat crossing points or virtual crossing points, then it must be trivial. From the geometric viewpoint, flat virtual knots can be seen as immersed curves on an oriented surface. The readers are referred to \cite{Tur2004} for further details of the homotopy classes and cobordism classes of flat virtual knots.

\begin{figure}[h]
\centering
\includegraphics[width=5cm]{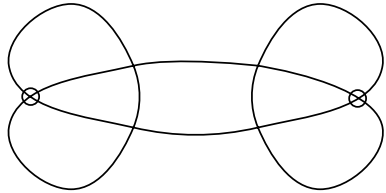}\\
\caption{Kishino flat virtual knot}\label{figure3}
\end{figure}

Let $K$ be a virtual knot diagram, we can define a flat virtual knot diagram $F(K)$ by replacing all real crossing points of $K$ with flat crossing points. It is not difficult to see that the flat virtual knot represented by $F(K)$ does not depend on the choice of the diagram $K$. In other words, if $K$ and $K'$ are related by a sequence of generalized Reidemeister moves, then $F(K)$ and $F(K')$ are related by a sequence of corresponding flat generalized Reidemeister moves. We shall call $F(K)$ a \emph{shadow} of $K$ and say $K$ \emph{overlies} $F(K)$. Conversely, if a flat virtual knot diagram $\widetilde{K}$ has $n$ flat crossing points, then there are totally $2^n$ virtual knot diagrams overlying $\widetilde{K}$. A simple but important fact is, if $\widetilde{K}$ is a nontrivial flat virtual knot then all these $2^n$ virtual knot diagrams represent nontrivial virtual knots. Moreover, if a flat virtual knot diagram has the minimal number of flat crossing points then all the virtual knots overlying it realize the minimal real crossing number. Figure \ref{figure3} provides an illustration of the Kishino flat virtual knot, which is known to be nontrivial \cite{Kau2012}. Therefore all the 16 Kishino virtual knot diagrams overlying it represent nontrivial virtual knots. This explains why the investigation of flat virtual knots is important but not easy in general.

\section{Two virtual knot invariants}\label{section3}
\subsection{An $\mathcal{M}_1^u$-valued virtual knot invariant}\label{3.1}\quad

Let $K$ be a virtual knot diagram and $c$ a real crossing point of it. We want to associate a unoriented flat virtual knot $\widetilde{K}_c$ to $c$ which satisfies the chord index axioms defined in Definition \ref{axiom}. There are two kinds of resolution of the crossing point $c$, see Figure \ref{figure4} (here $c$ could be a positive or a negative crossing). We use \emph{0-smoothing} to denote the resolution which preserves the number of components and \emph{1-smoothing} to denote the other one. After applying 0-smoothing at $c$ we will get another virtual knot, which is unoriented even if $K$ is oriented. We use $K_c$ to denote this unoriented virtual knot and use $\widetilde{K}_c$ to denote the shadow of it, i.e. $\widetilde{K}_c=F(K_c)$.
\begin{figure}[h]
\centering
\includegraphics{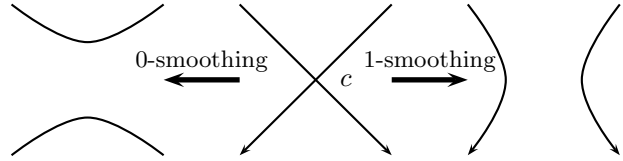}\\
\caption{Two resolutions of $c$}\label{figure4}
\end{figure}
\begin{theorem}\label{theorem3.1}
$\widetilde{K}_c$ satisfies the chord index axioms.
\end{theorem}
\begin{proof}
It suffices to show that $\widetilde{K}_c$ satisfies the five requirements in Definition \ref{axiom}.
\begin{enumerate}
  \item If $c$ is a crossing point appearing in $\Omega_1$, after performing 0-smoothing at $c$, it is easy to observe that $\widetilde{K}_c=F(K)$ without considering the orientation, which is a fixed element in $\mathcal{M}_1^u$.
  \item Consider the two crossing points in $\Omega_2$, say $c_1$ and $c_2$. As illustrated in Figure \ref{figure5}, in both cases the chord indices $\widetilde{K}_{c_1}$ and $\widetilde{K}_{c_2}$ are equivalent as flat virtual knots.
  \begin{figure}[h]
  \centering
  \includegraphics{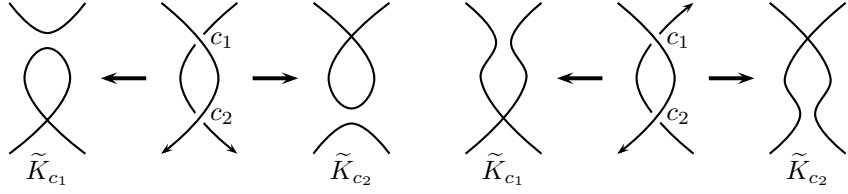}\\
  \caption{Resolutions of crossing points in $\Omega_2$}\label{figure5}
  \end{figure}
  \item Assume $K$ and $K'$ are related by one $\Omega_3$ move, we use $c_1, c_2, c_3$ to denote the three crossing points in $K$ and use $c_1', c_2', c_3'$ to denote the corresponding crossing points in $K'$. It is easy to see that $\widetilde{K}_{c_1}=\widetilde{K'}_{c_1'}, \widetilde{K}_{c_2}=\widetilde{K'}_{c_2'}$ and $\widetilde{K}_{c_3}=\widetilde{K'}_{c_3'}$ from Figure \ref{figure6}.
  \begin{figure}
  \centering
  \includegraphics{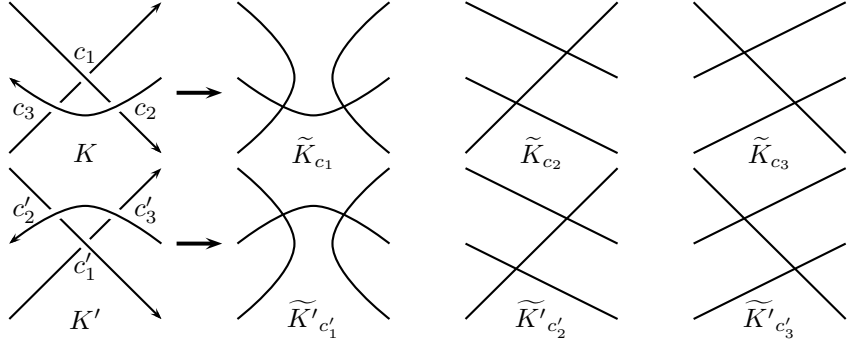}\\
  \caption{Resolutions of crossing points in $\Omega_3$}\label{figure6}
  \end{figure}
  \item For $\Omega^v_3$, there exist two possibilities: in one case two chord indices are the same, in the other case one chord index can be obtained from the other one by two $F\Omega_2'$-moves.
  \item If a crossing point $c$ is not involved in the move, then the two chord indices are related by a flat version of this move.
\end{enumerate}
\end{proof}

\begin{theorem}\label{theorem3.2}
Let $K$ be a virtual knot, then $\mathcal{F}(K)=\sum\limits_cw(c)\widetilde{K}_c-w(K)F(K)\in\mathcal{M}_1^u$ is a virtual knot invariant. Here $F(K)$ should be understood as the shadow of $K$ without the orientation, the sum runs over all the real crossing points of $K$, and $w(c), w(K)$ denote the writhe of $c$ and $K$ respectively.
\end{theorem}
\begin{proof}
Notice that $\widetilde{K}_c=F(K)$ if $c$ is the crossing point involved in $\Omega_1$ and the writhes of the two crossing points involved in $\Omega_2$ are distinct. The result follows directly from the chord index axioms.
\end{proof}

Obviously, if $K$ is a classical knot, then $\mathcal{F}(K)=0$. Actually, since any flat virtual knot diagram with one or two flat crossing points represent the unknot. It follows that if the number of real crossing points of $K$ is less than or equal to 2, we also have $\mathcal{F}(K)=0$.

\begin{proposition}
Let $K$ be a virtual knot diagram. If $F(K)$ is a minimal flat crossing number diagram, then all minimal real crossing number diagrams of $K$ have the same writhe.
\end{proposition}
\begin{proof}
As we mentioned in end of Section \ref{section2}, if $F(K)$ is a minimal flat crossing number diagram then $K$ must realizes the minimal number of real crossing points. Notice that any flat knot $\widetilde{K}_c$ has fewer flat crossing number than $F(K)$, it follows that the coefficient of $F(K)$ in $\mathcal{F}(K)$ equals $-w(K)$. Choose another minimal real crossing number diagram $K'$ which represents the same virtual knot as $K$. Note that $F(K)$ and $F(K')$ represent the same flat virtual knot, then the coefficients of them must be the same. It follows that $w(K)=w(K')$.
\end{proof}

\begin{corollary}
Let $\widetilde{K}$ be a minimal flat crossing number diagram of a flat virtual knot. If the number of flat crossing points in $\widetilde{K}$ equals $n$, then the $2^n$ virtual knot diagrams overlying $\widetilde{K}$ represent at least $n+1$ distinct virtual knots.
\end{corollary}

\begin{proposition}
Let $K$ be an oriented virtual knot diagram, if we use $r(K)$ to denote the diagram obtained from $K$ by reversing the orientation, and use $m(K)$ to denote the diagram obtained from $K$ by switching all real crossing points, then we have $\mathcal{F}(r(K))=\mathcal{F}(K)$ and $\mathcal{F}(m(K))=-\mathcal{F}(K)$.
\end{proposition}
\begin{proof}
Notice that as a unoriented flat virtual knot we have $F(r(K))=F(m(K))=F(K)$. On the other hand, the unoriented flat virtual knots $\widetilde{r(K)_c}$ and $\widetilde{m(K)_c}$ are equivalent to $\widetilde{K_c}$, but the writhe of $c$ is preserved in $r(K)$ but changed in $m(K)$.
\end{proof}

\begin{remark}
It is well known that the connected sum operation is not well defined in virtual knot theory. In fact, the result of a connected sum depends on the place where the operation is taken. However, we can still talk about the connected sum of two virtual knot diagrams. Let $K_1$ and $K_2$ be two separate virtual knot diagrams and $K_1\sharp K_2$ the connected sum of them with a fixed place where the connected-sum is taken. According to the definition of $\mathcal{F}(K)$, we have
\begin{flalign*}
\mathcal{F}(K_1\sharp K_2)=&\sum\limits_{c\in K_1\sharp K_2}w(c)\widetilde{K_1\sharp K_2}_c-w(K_1\sharp K_2)F(K_1\sharp K_2)\\
=&\sum\limits_{c\in K_1\sharp K_2}w(c)\widetilde{K_1\sharp K_2}_c-w(K_1)F(K_1)\sharp F(K_2)-w(K_2)F(K_1)\sharp F(K_2)\\
=&(\sum\limits_{c\in K_1}w(c)\widetilde{K_1}_c-w(K_1)F(K_1))\sharp F(K_2)+F(K_1)\sharp(\sum\limits_{c\in K_2}w(c)\widetilde{K_2}_c-w(K_2)F(K_2))\\
=&\mathcal{F}(K_1)\sharp F(K_2)+F(K_1)\sharp\mathcal{F}(K_2)
\end{flalign*}
where all the connected sum mentioned above are taken at the same place. Here we want to remark that the equality $\mathcal{F}(K_1\sharp K_2)=\mathcal{F}(K_1)\sharp F(K_2)+F(K_1)\sharp\mathcal{F}(K_2)$ does not hold for virtual knots in general. This is because even if two flat virtual knot diagrams are equivalent up to flat generalized Reidemeister moves, after taking the connected sum with another flat virtual knot diagram they may represent different flat virtual knots, see the following example. It means that $\mathcal{F}(K)$ can be used to show the fact that the connected sum operation is not well defined for virtual knots. Note that as a specialization, the writhe polynomial $W_{K}(t)$ (see Example \ref{3.9} for the definition) is additive with respect to connected sum \cite{Che2013}.
\end{remark}

\begin{example}
Let $K$ be the virtual knot described in Figure \ref{figure7}. Direct calculation shows that $\mathcal{F}(K)=\widetilde{\text{Kishino}}+4U-5F(K)$, where $\widetilde{\text{Kishino}}$ denotes the Kishino flat virtual knot depicted in Figure \ref{figure3} and $U$ denotes the unknot. Since $\widetilde{\text{Kishino}}$ is nontrivial, we conclude that $\mathcal{F}(K)\neq0$. Notice that $K$ can be regarded as the connected sum of two virtual trefoil knots, but $\mathcal{F}$ evaluated on the virtual trefoil knot equals zero.
\end{example}
\begin{figure}[h]
\centering
\includegraphics[width=5cm]{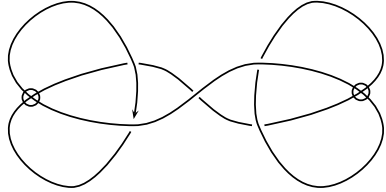}\\
\caption{A virtual knot $K$ with nontrivial $\mathcal{F}$}\label{figure7}
\end{figure}

\begin{remark}
The invariant $\mathcal{F}(K)$ can be easily extended to an invariant of $n$-component virtual links $(n\geq2)$. For this purpose, it suffices to redefine $\mathcal{M}_1^u$ to be the free $\mathbf{Z}$-module generated by all unoriented flat virtual links. Now the the chord index of a self-crossing point is a unoriented $n$-component flat virtual link, and the chord index of a mixed-crossing point is a unoriented $(n-1)$-component flat virtual link.
\end{remark}

\subsection{An $\mathcal{M}_2$-valued virtual knot invariant}\label{3.2}\quad

Let $K$ be a virtual knot diagram and $c$ a real crossing point of $K$. We plan to use the 1-smoothing operation to define an oriented 2-component flat virtual link $\widetilde{L}_c$ such that it also satisfies the chord index axioms defined in Definition \ref{axiom}. Recall that $\mathcal{M}_2$ is the free $\mathbf{Z}$-module generated by all the oriented  2-component flat virtual links. Performing 1-smoothing at $c$ transforms the virtual knot $K$ into an oriented 2-component virtual link. We use $L_c$ to denote this oriented virtual link and use $\widetilde{L}_c$ to denote the shadow of it, which lies in $\mathcal{M}_2$. Then we have the following result.
\begin{theorem}
$\widetilde{L}_c$ satisfies the chord index axioms.
\end{theorem}
\begin{proof}
Analogous to the proof of Theorem \ref{theorem3.1}, it suffices to check that $\widetilde{L}_c$ satisfies all the five conditions in Definition \ref{axiom}.
\begin{enumerate}
\item If $c$ is a crossing point involved in $\Omega_1$, then after the operation we see that $\widetilde{L}_c=F(K)\cup U$, which is a fixed element of $\mathcal{M}_2$, see Figure \ref{figure8}.
\begin{figure}[h]
\centering
\includegraphics{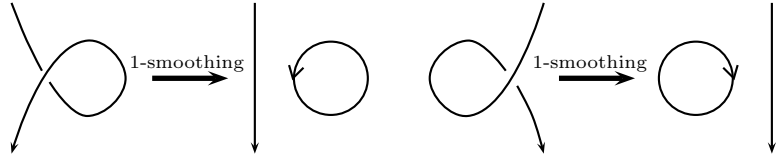}\\
\caption{Resolution of the crossing point in $\Omega_1$}\label{figure8}
\end{figure}
\item Let $c_1, c_2$ be the two crossing points in $\Omega_2$. It is easy to find that for each case we have $\widetilde{L}_{c_1}=\widetilde{L}_{c_2}$, see Figure \ref{figure9}.
\begin{figure}[h]
\centering
\includegraphics{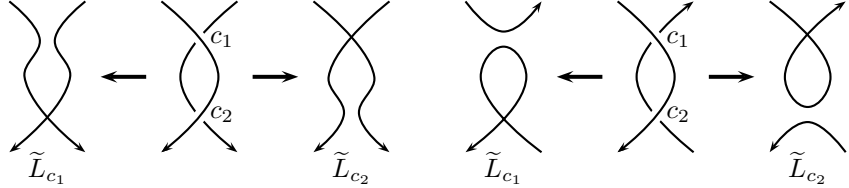}\\
\caption{Resolutions of the two crossing points in $\Omega_2$}\label{figure9}
\end{figure}
\item For $\Omega_3$, let us denote the three crossing points before the move by $c_1, c_2, c_3$, and use $c_1', c_2', c_3'$ to denote the corresponding crossing points after the move, see Figure \ref{figure10}. Applying 1-smoothing for each of them, it follows evidently that $\widetilde{L}_{c_i}=\widetilde{L'}_{c_i'}$ for all $1\leq i\leq 3$.
\begin{figure}[h]
\centering
\includegraphics{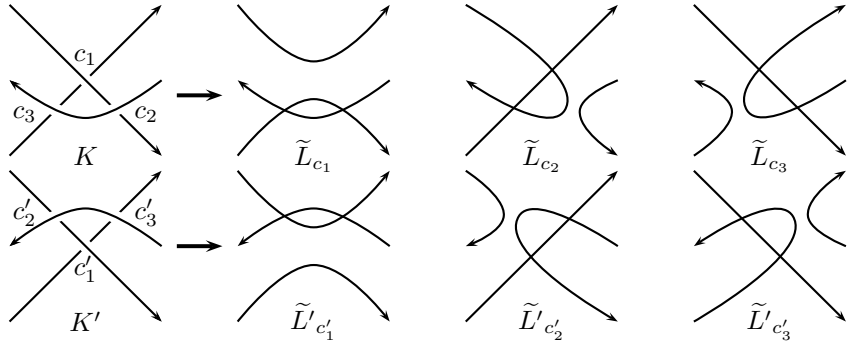}\\
\caption{Resolutions of the three crossing points in $\Omega_3$}\label{figure10}
\end{figure}
\item As before, there are two possibilities according to the orientations of the strands in $\Omega^v_3$: in one case the two chord indices are the same, and in the other case one chord index can be obtained from the other one by two $F\Omega_2'$-moves.
  \item If $c$ does not appear in the move, then the two chord indices can be connected by one flat version of the move.
\end{enumerate}
\end{proof}

Similar to Theorem \ref{theorem3.2}, this chord index also provides us a virtual knot invariant.

\begin{theorem}
Let $K$ be a virtual knot, then $\mathcal{L}(K)=\sum\limits_cw(c)\widetilde{L}_c-w(K)(F(K)\cup U)\in\mathcal{M}_2$ defines a virtual knot invariant. Here $F(K)$ denotes the shadow of $K$, the sum runs over all the real crossing points of $K$ and $w(c), w(K)$ denote the writhe of $c$ and $K$ respectively.
\end{theorem}

\begin{remark}
Here each $\widetilde{L}_c$ should be understood as a unordered 2-component flat virtual link. According to Figure \ref{figure8}, one can observe that there does not exist a preferred order on the two components of $\widetilde{L}_c$.
\end{remark}

The following proposition follows directly from the definition of $\mathcal{L}(K)$.
\begin{proposition}
Let $K$ be an oriented virtual knot diagram, if we use $r(K)$ to denote the diagram obtained from $K$ by reversing the orientation, and $m(K)$ denotes the diagram obtained from $K$ by switching all real crossing points, then we have $\mathcal{L}(r(K))=r(\mathcal{L}(K))$ and $\mathcal{L}(m(K))=-\mathcal{L}(K)$.
\end{proposition}

It is evident that $\mathcal{L}(K)=0$ if $K$ is a classical knot. But unlike $\mathcal{F}(K)$, which vanishes on virtual knots with two real crossing points, $\mathcal{L}(K)$ is able to distinguish the virtual trefoil knot from the unknot. Note that up to some symmetries, the virtual trefoil knot is the only one which has real crossing number two.

\begin{example}\label{3.11}
Consider the virtual trefoil knot $K$ in Figure \ref{figure2}. We have $\mathcal{L}(K)=2\widetilde{HL}-2(U\cup U)$, here $\widetilde{HL}$ denotes the flat virtual Hopf link which can be obtained from the classcial Hopf link diagram by replacing the two real crossing points with one virtual crossing point and one flat crossing point. As before, $U$ means the unknot. In order to prove $\mathcal{L}(K)\neq0$, it suffices to show that $\widetilde{HL}$ is nontrivial. To this end, let us consider the \emph{flat linking number} of a unordered 2-component flat virtual link. Suppose $\widetilde{L}=\widetilde{K}_1\cup\widetilde{K}_2$ is a unordered 2-component flat virtual link. Denote $C_{\widetilde{K}_1\cap\widetilde{K}_2}$ to be the set of all flat crossing points between $\widetilde{K}_1$ and $\widetilde{K}_2$. Replacing all the flat crossing points in $C_{\widetilde{K}_1\cap\widetilde{K}_2}$ with a real crossing point such that at each crossing point the over-strand belongs to $\widetilde{K}_1$, then we can define the flat linking number as
\begin{center}
$lk(\widetilde{L})=|\sum\limits_{c\in C_{\widetilde{K}_1\cap\widetilde{K}_2}}w(c)|$,
\end{center}
where $w(c)$ means the writhe of the real crossing point $c$. Clearly, this definition does not depend on the order of $\widetilde{K}_1$ and $\widetilde{K}_2$, and it is invariant under all flat generalized Reidemeister moves. Obviously, $lk(\widetilde{HL})=1$ and it follows that $\mathcal{L}(K)\neq0$.
\end{example}

\begin{example}
Consider the virtual knot $K$ illustrated in Figure \ref{figure15}, which is a variant of the Kishino virtual knot. Direct calculation shows that $\mathcal{L}(K)=\widetilde{L}_{c_1}-\widetilde{L}_{c_2}-\widetilde{L}_{c_3}+\widetilde{L}_{c_4}\neq0$. The last inequality follows from the fact that neither of $\{\widetilde{L}_{c_i}\}_{1\leq i\leq4}$ is trivial, and any pair of them are inequivalent. One way to see this is to extend the writhe polynomial from virtual knots to virtual links, which can be found in the forthcoming paper \cite{Xu2018}.
\begin{figure}
\centering
\includegraphics{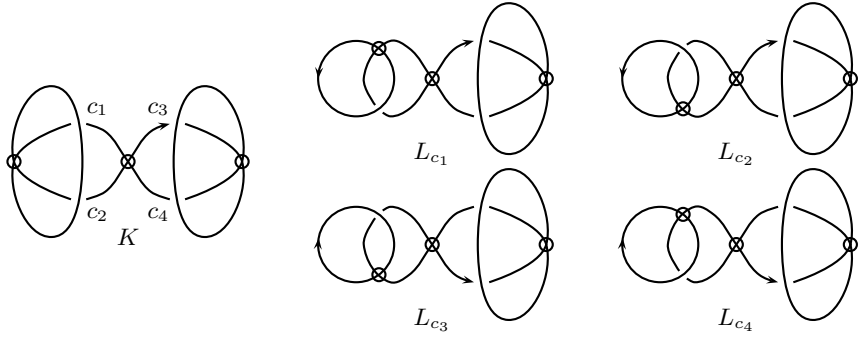}\\
\caption{A virtual knot $K$ with nontrivial $\mathcal{L}$}\label{figure15}
\end{figure}
\end{example}

\begin{remark}
We remark that there is no difficulty in extending this invariant $\mathcal{L}(K)$ to an invariant of virtual links. One just needs to make some suitable modifications on the definition of $\mathcal{M}_2$.
\end{remark}

\subsection{Geometric interpretations of $\mathcal{F}(K)$ and $\mathcal{L}(K)$}\quad

In this subsection we give a geometric interpretation for $\mathcal{F}(K)$. Recall that each virtual knot can be considered as an embedded curve in $\Sigma_g\times[0, 1]$, or equivalently, a knot diagram on $\Sigma_g$. By ignoring the over/undercrossing information of each crossing point, we can realize each flat virtual knot as a generic immersed closed curve on $\Sigma_g$. Here we follow the construction of abstract knots, which is known to be equivalent to the set of virtual knots \cite{Kam2000}, to construct a knot diagram on a closed oriented surface.

Let $K$ be a virtual knot diagram and $F(K)$ the shadow of $K$. Consider $F(K)$ as an abstract 4-valent graph as follows. If the number of flat crossing points equals zero, then the graph is just a circle. Otherwise, the set of vertices consists of all flat crossing points and each vertex has valency four. By ignoring all the virtual crossing points, we obtain a abstract 4-valent graph from $F(K)$. In particular, this is a 4-valent planar graph if there is no virtual crossing points. If there exist some virtual crossing points, this is an immersed 4-valent graph in the plane. After thickening vertices into disjoint small disks and thickening each edge into a ribbon connecting the corresponding two disks we obtain a compact oriented immersed surface in $R^2$, or an embedded surface in $R^3$. Let us use $\Sigma_K$ to denote this surface and $\Sigma_g$ the closed surface obtained from $\Sigma_K$ by attaching disks to the boundaries of $\Sigma_K$. It is evident to find that if the number of flat crossing points equals $n$ and the number of boundaries of $\Sigma_K$ equals $m$, then the genus $g$ of $\Sigma_g$ is equal to $\frac{m-n}{2}+1$. It is worth noting that this is the minimum genus among all the surfaces in which $\Sigma_K$ is embeddable. Replacing each flat crossing point with the original real crossing point of $K$ yields a knot diagram on the closed surface $\Sigma_g$.

Let us still use $K$ to denote the knot diagram on $\Sigma_g$. Choose a crossing point $c$, we can perform the 0-smoothing on the surface. After that, we obtain another unoriented knot diagram $K_c$ and the corresponding shadow $\widetilde{K}_c$. Now the invariant $\mathcal{F}(K)=\sum\limits_cw(c)\widetilde{K}_c-w(K)F(K)$ can be regarded as a linear combination of some unoriented flat virtual knot diagrams on $\Sigma_g$. Similarly, one can also consider the invariant $\mathcal{L}(K)$ in this way. The benefit of this geometric interpretation is, for some specialization (for instance the Example \ref{3.9} below), the invariant can be defined in terms of the homological intersection form $H_1(\Sigma_g)\times H_1(\Sigma_g)\rightarrow\mathbf{Z}$.

\subsection{Some special cases of $\mathcal{F}(K)$ and $\mathcal{L}(K)$}\label{3.3}\quad

In Subsection \ref{3.1} and  \ref{3.2}, we introduced two virtual knot invariants valued in $\mathcal{M}_1^u$ and $\mathcal{M}_2$ respectively. Combining them with a concrete flat virtual knot/link invariant we obtain a concrete and sometimes much more usable virtual knot invariant. In this subsection we show that several invariants appeared in the literature in recent years can be regarded as special cases of our invariants defined in Subsection \ref{3.1} and Subsection \ref{3.2}.

\begin{example}\label{3.9}
The first example is the writhe polynomial $W_K(t)$ introduced by the first and second author in \cite{Che2013}, which generalizes the index polynomial defined by Henrich in \cite{Hen2010} and the odd writhe polynomial defined by the first author in \cite{Che2014}. The writhe polynomial was also independently defined (with different names) by Dye in \cite{Dye2013}, Kauffman in \cite{Kau2013}, Im, Kim, Lee in \cite{Im2013}, and Satoh, Taniguchi in \cite{Sat2014}. The key point of the definition of the $W_K(t)$ is the chord index Ind$(c)$ introduced in \cite{Che2014}, which assigns an integer to each real crossing point of a virtual knot diagram and satisfies the chord index axioms in Definition \ref{axiom}. Here we follow the approach of Folwaczny and Kauffman \cite{Fol2013} with a little modifications.

Assume $K$ is a virtual knot diagram and $c$ is a real crossing point of it. By applying the 1-smoothing operation to $c$ we will get a 2-component virtual link $L_c=K_1\cup K_2$. The order of these two components is arranged as follows: consider the crossing point $c$ in Figure \ref{figure4}, if $c$ is positive then we call the component on the left side $K_1$ and the component on the right side $K_2$; conversely, if $c$ is negative then we use $K_1$ to denote the component on the right side and use $K_2$ to denote the component on the left side. Consider all the real crossing points between $K_1$ and $K_2$, which can be divided into two sets $C_{12}$ and $C_{21}$. Here $C_{12}$ denotes the set of real crossing points where the over-strands belong to $K_1$ and $C_{21}$ those the over-strands belong to $K_2$. Now we can define the \emph{index} of $c$ as
\begin{center}
Ind$(c)=\sum\limits_{c\in C_{12}}w(c)-\sum\limits_{c\in C_{21}}w(c)$
\end{center}
and the \emph{writhe polynomial} as
\begin{center}
$W_K(t)=\sum\limits_cw(c)t^{\text{Ind}(c)}-w(K)$.
\end{center}
According to the definition of the index, it is evident that Ind$(c)$ is invariant under the crossing change of any real crossing point of $L_c$. Hence it is an invariant of $\widetilde{L}_c$, and $W_K(t)$ is a specialization of $\mathcal{L}(K)$.

As we mentioned before, the writhe polynomial can be topologically defined via the homological intersection of the supporting surface $\Sigma_g$. As an application, by using this topological interpretation Boden, Chrisman and Gaudreau proved that the writhe polynomial is a concordance invariant of virtual knots. See \cite{Bod2018} for more details.
\end{example}

\begin{example}\label{3.14}
The second example concerns a sequence of 2-variable polynomial invariants $L_K^n(t, l)$ recently introduced by Kaur, Prabhakar and Vesnin in \cite{Kau2018}. We will show that this sequence of polynomial invariants combines some information of $\mathcal{F}(K)$ and $\mathcal{L}(K)$.

Let $K$ be a virtual knot diagram and $c$ a real crossing point of $K$. Assume the writhe polynomial $W_K(t)=\sum\limits_na_nt^n$, it is known that $W_K(t)-W_K(t^{-1})$ is a polynomial invariant of $F(K)$, see for example \cite{Tur2004,Che2016,Kau2018}. Now $W_K(t)-W_K(t^{-1})$ has the form $\sum\limits_n(a_n-a_{-n})t^n$, and the coefficient of $t^n$ equals $a_n-a_{-n}$, which is called the \emph{$n$-th dwrithe} in \cite{Kau2018}, is also a flat virtual knot invariant. Following \cite{Kau2018}, let us use $\nabla J_n(K)$ to denote $a_n-a_{-n}$. Now we can assign an index $t^{\text{Ind}(c)}l^{|\nabla J_n(K_c)|}$ to $c$, as before here $K_c$ is referred to the (unoriented) virtual knot obtained from $K$ by applying 0-smoothing at $c$. In order to calculate $\nabla J_n(K_c)$ one needs to choose an orientation for $K_c$, but the abstract value guarantees that the index is well defined. It is worth noting that this index also satisfies the chord axioms, and if a crossing point is involved in $\Omega_1$ then the index equals $l^{|\nabla J_n(K)|}$. Now the polynomials $L_K^n(t, l)$ are defined as below
\begin{center}
$L_K^n(t, l)=\sum\limits_cw(c)t^{\text{Ind}(c)}l^{|\nabla J_n(K_c)|}-w(K)l^{|\nabla J_n(K)|}$.
\end{center}
Some examples are given in \cite{Kau2018} to show that $L_K^n(t, l)$ can distinguish two virtual knots with the same writhe polynomial. As we mentioned above, Ind$(c)$ is a flat virtual knot invariant of $L_c$ and $|\nabla J_n(K_c)|$ is a flat virtual knot invariant of $K_c$, hence $L_K^n(t, l)$ can be viewed as a mixture of some information coming from $\mathcal{F}(K)$ and $\mathcal{L}(K)$.
\end{example}

We end this subsection with a non-example. We will give an index type virtual knot invariant and prove that this index cannot be obtained from $\mathcal{F}(K)$ or $\mathcal{L}(K)$.

\begin{example}\label{3.18}
Before the discovery of Ind$(c)$, it was first noted by Kauffman that the parity of Ind$(c)$ plays an important role in virtual knot theory \cite{Kau2004}. Later Manturov proposed the parity axioms in \cite{Man2010} and proved that the parity projection $P_2$ which maps a virtual knot diagram to another virtual knot diagram is well defined. Here the projection $P_2$ is defined as follows: assume we have a virtual knot diagram $K$, then we can assign an index Ind$(c)$ to a real crossing point $c$ (see Example \ref{3.9}). After replacing all the real crossing points that have odd indices with virtual crossing points, we obtain a new virtual knot diagram $P_2(K)$. The key point is that, if $K$ and $K'$ are related by finitely many generalized Reidemeister moves, so are $P_2(K)$ and $P_2(K')$. It means that $P_2$ defines a well defined projection from virtual knots to virtual knots.

This projection can be easily generalized to a sequence of projections $P_n$ for any nonnegative integer $n$. Let $K$ be a virtual knot diagram, we define $P_n(K)$ to be the virtual knot diagram obtained from $K$ by replacing all real crossing points whose indices are not multiples of $n$ with virtual crossing points. From the viewpoint of Gauss diagram, one just leaves all the chords whose indices are multiples of $n$ and delete all other chords. It was proved that in $\Omega_3$ the index of one crossing point equals the sum of the indices of the other two crossing points \cite{Che2017}. With this fact, it is not difficult to observe that $P_n$ provides a well defined projection from the set of virtual knots to itself. For example, $P_1=id$ and $P_2$ coincides with Manturov's parity projection mentioned above. Now each $P_n(K)$ can be regarded as a virtual knot invariant of $K$, and all $\{P_n(K)\}$ together provides a sequence of virtual knot invariants. Note that for each $K$, only finitely many of $\{P_n(K)\}$ are nontrivial. For example, if $K$ is a classical knot, then $P_n(K)=K$ for any nonnegative integer $n$. For a given virtual knot invariant $f$, the sequence $\{f(P_n(K))\}$ gives rise to a family of virtual knot invariants. In particular, one can consider $P_{n_k}\circ P_{n_{k-1}}\circ\cdots\circ P_{n_1}(K)$ for any nonnegative sequence $(n_1, \cdots, n_k)$. As another example, the case of $f=W_K(t)$ was discussed by Im and Kim in \cite{Im2017}.

Let us consider $W_{P_0(K)}(t)$. We are planning to show that the index of this invariant, denotes by Ind$_0(c)$, cannot be recovered from $\widetilde{K}_c$ or $\widetilde{L}_c$. First, we need to reinterpret $W_{P_0(K)}(t)$ as an index type polynomial. Let $K$ be a virtual knot diagram and $C_0(K)$ the set of real crossing points with index zero. We use $G(K)$ to denote the Gauss diagram of $K$ and use $G_0(K)$ to denote the Gauss diagram obtained by removing all chords with nonzero indices. Choose a crossing point $c\in C_0(K)$, we will also use the same symbol to denote the corresponding chord in $G_0(K)$. Let $c_+, c_-$ be the two endpoints of the chord $c$, such that the chord is directed from $c_+$ to $c_-$. For any other chord in $G_0(K)$, we also associate $\pm1$ to the endpoints of it such that $+1$ $(-1)$ corresponds to the overcrossing (undercrossing) point. Now $c_+$ and $c_-$ divides the big circle of $G_0(K)$ into two open arcs, one is directed from $c_+$ to $c_-$ and the other one is directed from $c_-$ to $c_+$. Let us use $\widehat{c_+c_-}$ and $\widehat{c_-c_+}$ to denote these two open arcs respectively. Now the chords in $C_0(K)$ which have nonempty intersections with $c$ can be divided into two sets $C_c^{\rightarrow}$ and $C_c^{\leftarrow}$, where $c'\in C_c^{\rightarrow}$ $(C_c^{\leftarrow})$ if and only if $c'_+\in\widehat{c_+c_-}$ $(\widehat{c_-c_+})$ and $c'_-\in\widehat{c_-c_+}$ $(\widehat{c_+c_-})$. Now the index Ind$_0(c)$ can be defined as
\begin{center}
Ind$_0(c)=\sum\limits_{c'\in C_c^{\rightarrow}}w(c')-\sum\limits_{c'\in C_c^{\leftarrow}}w(c')$.
\end{center}
One can check that this definition coincides with the index (in the sense of Example \ref{3.9}) of $c$ in the virtual knot corresponding the $G_0(K)$. It follows that
\begin{center}
$W_{P_0(K)}(t)=\sum\limits_{c\in C_0(K)}w(c)(t^{\text{Ind}_0(c)}-1)$.
\end{center}

\begin{figure}
\centering
\includegraphics{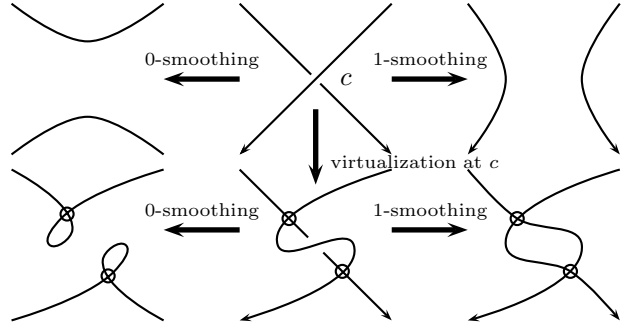}\\
\caption{Virtualization at $c$}\label{figure11}
\end{figure}
In order to show that Ind$_0(c)$ cannot be recovered from $\widetilde{K}_c$ or $\widetilde{L}_c$, it suffices to construct two virtual knot diagrams $K$ and $K'$, such that there exist two real crossing points $c\in K, c'\in K'$ which satisfy $\widetilde{K_c}=\widetilde{K'_{c'}}, \widetilde{L_c}=\widetilde{L'_{c'}}$ but Ind$_0(c)\neq\text{Ind}_0(c')$. The key observation is that $\widetilde{K_c}$ and $\widetilde{L_c}$ are both preserved under a virtualization at $c$, see Figure \ref{figure11}. We remark that the virtualization operation was first used by Kauffman to construct infinitely many nontrivial virtual knots with unit Jones polynomial. Although the virtualization operation at $c$ preserves both $\widetilde{K_c}$ and $\widetilde{L_c}$, it usually changes the indices of other crossing points and hence changes Ind$_0(c)$. As a concrete example, consider the two Gauss diagrams $G(K)$ and $G(K')$ depicted in Figure \ref{figure12}. One can easily find that $\widetilde{K_c}=\widetilde{K'_{c'}}$ and $\widetilde{L_c}=\widetilde{L'_{c'}}$, but Ind$_0(c)=1\neq0=\text{Ind}_0(c')$.
\begin{figure}
\centering
\includegraphics{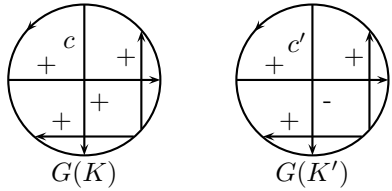}\\
\caption{$G(K)$ and $G(K')$}\label{figure12}
\end{figure}
\end{example}

The main reason why Ind$_0(c)$ cannot be recovered from $\widetilde{K}_c$ or $\widetilde{L}_c$ is that the notion of chord index is twice used in the definition of Ind$_0(c)$. This idea was first used by Turaev in flat virtual knot theory \cite{Tur2004} and later by M.-J. Jeong in the definition of zero polynomial for virtual knots \cite{Jeo2016}. Recently the zero polynomial was extended to a transcendental function invariant of virtual knots in \cite{Che2017}.

\section{New invariants revisited from the viewpoint of finite type invariant}\label{section4}
Recall that a finite type invariant (also called Vassiliev invariant) of degree $n$ is a (virtual) knot invariant valued in an abelian group which vanishes on all singular knots with $k$ singularities provided that $k\geq n+1$. More precisely, for a given abelian group $A$, if $f:K\rightarrow A$ is a virtual knot invariant which associates each virtual knot with an element of $A$. Then we can extend $f$ from virtual knots to singularities virtual knots via the following recursive relation
\begin{center}
$f^{n}(K)=f^{n-1}(K_+)-f^{n-1}(K_-)$,
\end{center}
here $K_+(K_-)$ is obtained from $K$, a singular virtual knot with $n$ singularities, by replacing a singular point with a positive (negative) crossing point. In particular, we set $f^{(0)}=f$ as the initial condition. Now we say $f$ is a \emph{finite type invariant of degree $n$} if $f^{(n+1)}(K)=0$ for any singular virtual knot $K$ with $n+1$ singularities, and there exists a singular virtual knot $K$ with $n$ singularities which satisfies $f^{(n)}(K)\neq0$. For classical knots, this definition coincides with the definition given by Birman and Lin in \cite{Bir1993}. We would like to remark that for virtual knots there exists another kind of finite type invariants. We refer the reader to \cite{Gou2000} for more details.

The main result of this section is the following:
\begin{theorem}
Both $\mathcal{F}(K)$ and $\mathcal{L}(K)$ are finite type invariants of degree one.
\end{theorem}
\begin{proof}
We first prove that $\mathcal{F}(K)$ is a finite type invariant of degree one. For this purpose, we need to show that $\mathcal{F}^{(2)}$ vanishes on any singular virtual knot with two singularities and there is a singular virtual knot with one singularity which has nontrivial $\mathcal{F}^{(1)}$.

Let $K$ be a virtual knot diagram and $c_1, c_2$ be two real crossing points. Without loss of generality, we assume that $w(c_1)=w(c_2)=+1$. We use $K_{-+}(K_{+-})$ to denote the virtual knot diagram obtained from $K$ by switching $c_1(c_2)$, and use $K_{--}$ to denote the diagram obtained from $K$ by switching both $c_1$ and $c_2$. Obviously, $K_{++}=K$. What we need to do is proving that $\mathcal{F}(K_{++})-\mathcal{F}(K_{+-})-\mathcal{F}(K_{-+})+\mathcal{F}(K_{--})=0$.

One computes
\begin{flalign*}
&\mathcal{F}(K_{++})-\mathcal{F}(K_{+-})-\mathcal{F}(K_{-+})+\mathcal{F}(K_{--})\\
=&(\sum\limits_cw(c)\widetilde{K_{++}}_c-w(K_{++})F(K_{++}))-(\sum\limits_cw(c)\widetilde{K_{+-}}_c-w(K_{+-})F(K_{+-}))\\
&-(\sum\limits_cw(c)\widetilde{K_{-+}}_c-w(K_{-+})F(K_{-+}))+(\sum\limits_cw(c)\widetilde{K_{--}}_c-w(K_{--})F(K_{--}))\\
=&\sum\limits_cw(c)\widetilde{K_{++}}_c-\sum\limits_cw(c)\widetilde{K_{+-}}_c-\sum\limits_cw(c)\widetilde{K_{-+}}_c+\sum\limits_cw(c)\widetilde{K_{--}}_c\\
=&(\widetilde{K_{++}}_{c_1}+\widetilde{K_{++}}_{c_2})-(\widetilde{K_{+-}}_{c_1}-\widetilde{K_{+-}}_{c_2})-(-\widetilde{K_{-+}}_{c_1}+\widetilde{K_{-+}}_{c_2})+(-\widetilde{K_{--}}_{c_1}-\widetilde{K_{--}}_{c_2})\\
=&(\widetilde{K_{++}}_{c_1}-\widetilde{K_{--}}_{c_1})+(\widetilde{K_{++}}_{c_2}-\widetilde{K_{--}}_{c_2})+(\widetilde{K_{-+}}_{c_1}-\widetilde{K_{+-}}_{c_1})+(\widetilde{K_{+-}}_{c_2}-\widetilde{K_{-+}}_{c_2})\\
=&0+0+0+0\\
=&0
\end{flalign*}

On the other hand, consider the Kishino virtual knot $K$ which is obtained from the Kishino flat virtual knot in Figure \ref{figure3} by replacing all flat crossing points with positive crossing points. Denote the positive crossing point on the top left of $K$ by $c$. We still use $K_+$ to denote $K$ and use $K_-$ to denote the diagram after switching $c$. Then we have
\begin{flalign*}
&\mathcal{F}(K_{+})-\mathcal{F}(K_{-})\\
=&(\sum\limits_cw(c)\widetilde{K_{+}}_c-w(K_{+})F(K_{+}))-(\sum\limits_cw(c)\widetilde{K_{-}}_c-w(K_{-})F(K_{-}))\\
=&(4U-4\widetilde{\text{Kishino}})-(2U-2\widetilde{\text{Kishino}})\\
=&2U-2\widetilde{\text{Kishino}}\\
\neq&0
\end{flalign*}

For $\mathcal{L}(K)$, by mimicking the proof of $\mathcal{F}(K)$, one can also prove that $\mathcal{L}^{(2)}(K)$ vanishes on any singular virtual knot with two singular points. On the other hand, as a specialization of $\mathcal{L}(K)$, we know that the writhe polynomial $W_K(t)$ is a finite type invariant of degree one \cite{Dye2013,Che2016}. Hence there exists a singular virtual knot with one singular point which has nontrivial $\mathcal{L}(K)$. As a concrete example, one can consider the singular virtual trefoil, which can be obtained from the virtual trefoil in Figure \ref{figure2} by replacing one real crossing point with a singular point.
\end{proof}

The proof of the theorem above tells us that if a chord index of a crossing point is preserved when one switches any other crossing point, then this chord index cannot provide a nontrivial invariant for classical knots. Since for classical knots, there is no finite type invariant of degree one.

We know that the writhe polynomial is a finite type invariant of degree one. However, one can use the notion of chord index to define an elaborate virtual knot invariant which is of degree two. See \cite{Chr2014} for an example. We have shown that the writhe polynomial $W_K(t)$ can be recovered from $\mathcal{L}(K)$. At present, we have no idea whether it is possible to define a finite type invariant of higher degree which takes values in $\mathcal{M}_1^u$ or $\mathcal{M}_2$.

\section{From flat virtual knots to flat virtual knots}\label{section5}
In Section \ref{section3} we introduced two virtual knot invariants $\mathcal{F}(K)$ and $\mathcal{L}(K)$, which take values in $\mathcal{M}_1^u$ and $\mathcal{M}_2$ respectively. However, in general it is still not easy to distinguish two elements in $\mathcal{M}_1^u$ or $\mathcal{M}_2$. To this end, in this section we extend the main idea of Section \ref{section3} from virtual knots to flat virtual knots. It turns out that we can define two invariants for flat virtual knots, which take values in $\mathcal{M}_1$ and $\mathcal{M}_2$ respectively. Here $\mathcal{M}_1$ denotes the free $\mathbf{Z}$-module generated by all oriented flat virtual knots, and as before, $\mathcal{M}_2$ is referred to the free $\mathbf{Z}$-module generated by all oriented 2-component flat virtual links.

Given an oriented flat virtual knot diagram $\widetilde{K}$, similar to virtual knots, we can define a Gauss diagram $G(\widetilde{K})$ corresponding to $\widetilde{K}$. At first glance, since there is no over/undercrossing information we can just connect the two preimages of each flat crossing with a chord. So unlike the Gauss diagrams of virtual knots, now there is no directions or signs on each chord. However, we can still assign a direction to each chord as follows: replace each flat crossing point with a positive crossing point and then add an arrow to each chord in $G(\widetilde{K})$, directed from the preimage of the overcrossing to the preimage of the undercrossing. Now we obtain a Gauss diagram $G(\widetilde{K})$ of which each chord has a direction but no signs. The following result is well-known, see for example \cite{Tur2004}.

\begin{lemma}
Each $G(\widetilde{K})$ corresponds to a unique flat virtual knot $\widetilde{K}$.
\end{lemma}
\begin{proof}
Adding a positive sign to each chord of $G(\widetilde{K})$ gives us a Gauss diagram $G(K)$ of a positive virtual knot diagram $K$. By replacing all the positive crossing points of $K$ with flat crossing points we will obtain a flat virtual knot $\widetilde{K}$. However, $G(K)$ corresponds to infinitely many different virtual knot diagrams. We need to show that the shadow of them give rise to the same flat virtual knot. Actually, choose another virtual knot diagram $K'$ which also corresponds to $G(K)$. It is known that $K$ and $K'$ can be connected by finitely many $\Omega_1', \Omega_2', \Omega_3'$ and $\Omega_3^v$ moves. It follows that the shadows $F(K)$ and $F(K')$ can be connected by a sequence of $F\Omega_1', F\Omega_2', F\Omega_3', F\Omega_3^v$ moves. This completes the proof.
\end{proof}

Although there is no signs on flat crossing point, we can use the sign of the index to make up for that. More precisely, for a given flat virtual knot diagram $\widetilde{K}$, let us consider the positive virtual knot diagram over it, say $K^+$. Now we can use the index of Example \ref{3.9} to assign an integer Ind$(c)$ to each positive crossing point of $K^+$. For the original flat crossing point $c$ (for simplicity here we use the same symbol) in $\widetilde{K}$, we define the writhe of the flat crossing point $c$ as follows
\begin{center}
$w(c)=\text{sgn}(\text{Ind}(c))=
\begin{cases}
\frac{|\text{Ind}(c)|}{\text{Ind}(c)}& \text{if }\text{Ind}(c)\neq0;\\
0& \text{if }\text{Ind}(c)=0.
\end{cases}$
\end{center}
With the help of writhe on each flat crossing points, one can similarly define the writhe of a flat virtual knot as $w(\widetilde{K})=\sum\limits_cw(c)$, where the sum runs over all flat crossing points.

\begin{lemma}\label{lemma5.2}
$w(\widetilde{K})$ is a flat virtual knot invariant.
\end{lemma}
\begin{proof}
It is sufficient to notice that the flat crossing point involved in $F\Omega_1$ has writhe zero, the two flat crossing points involved in $F\Omega_2$ have opposite writhes and the writhe of each flat crossing point in $F\Omega_3$ is preserved under $F\Omega_3$. It is worth pointing out that it is possible that even if two flat virtual knot diagrams are related by a $F\Omega_3$ move, the corresponding positive virtual knot diagrams cannot be connected by a $\Omega_3$ move. However, in this case the writhe of each flat crossing point is still preserved. Obviously, the writhe of the flat crossing point in $F\Omega_3^v$ is also preserved under $F\Omega_3^v$. The proof is finished.
\end{proof}

\begin{example}\label{5.3}
As an example, let us consider the flat virtual knot $\widetilde{K_{p,q}}$ $(p, q\in\mathbf{N})$ whose lattice-like Gauss diagram consists of $p$ parallel horizontal chords directed from left to right and $q$ parallel vertical chords directed from top to bottom, such that each horizontal chord and each vertical chord have one intersection point. Clearly, every horizontal chord has index $-q$ and every vertical chord has index $p$. Thus, if $p\neq q$ we have $w(\widetilde{K_{p,q}})=-p+q\neq0$ and therefore $\widetilde{K_{p,q}}$ is nontrivial. We remark that if $p=q$ then $w(\widetilde{K_{p,q}})=0$, but in this case the flat virtual knot $\widetilde{K_{p,q}}$ is also nontrivial \cite{Tur2004}.
\end{example}

\begin{remark}
Replacing the writhe with index, one can define Ind$(\widetilde{K})=\sum\limits_c\text{Ind}(c)$. It turns out that Ind$(\widetilde{K})$ also defines an invariant for flat virtual knots. Actually, this is a trivial invariant, i.e. Ind$(\widetilde{K})=0$ for any flat virtual knot $\widetilde{K}$. One can easily conclude this from the definition of Ind$(c)$.
\end{remark}

Now we want to enhance the invariant $w(\widetilde{K})$ to a weighted sum by associating a chord index to each flat crossing point of $\widetilde{K}$. Let $c$ be a flat crossing point of $\widetilde{K}$. Similar as before, we can perform 0-smoothing or 1-smoothing to resolve this flat crossing point to get a unoriented flat virtual knot $\widetilde{K}_c$ or an oriented 2-component flat virtual link $\widetilde{L}_c$. In particular, for $\widetilde{K}_c$ now we fix an orientation on it, see Figure \ref{figure14}.

We define $\widetilde{\mathcal{F}}(\widetilde{K})=\sum\limits_cw(c)\widetilde{K}_c\in\mathcal{M}_1$ and $\widetilde{\mathcal{L}}(\widetilde{K})=\sum\limits_cw(c)\widetilde{L}_c\in\mathcal{M}_2$.
\begin{figure}[h]
\centering
\includegraphics{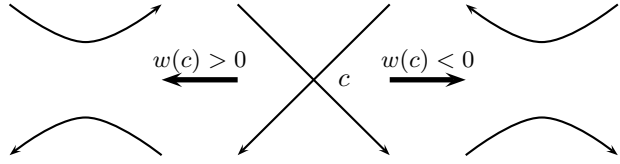}\\
\caption{The orientation of $\widetilde{K}_c$}\label{figure14}
\end{figure}

\begin{theorem}
$\widetilde{\mathcal{F}}(\widetilde{K})$ and $\widetilde{\mathcal{L}}(\widetilde{K})$ are both flat virtual knot invariants.
\end{theorem}
\begin{proof}
The proof of Lemma \ref{lemma5.2} told us the behavior of the writhe of each flat crossing point under the flat generalized Reidemeister moves. Notice that if the writhe of a flat crossing point equals zero, then it has no contribution to $\widetilde{\mathcal{F}}(\widetilde{K})$ or $\widetilde{\mathcal{L}}(\widetilde{K})$, although Figure \ref{figure14} did not tell us how to associate an orientation to $\widetilde{K}_c$ in this case. The rest of the proof is a routine check of the flat versions of Figure \ref{figure5}, Figure \ref{figure6}, Figure \ref{figure8}, Figure \ref{figure9} and Figure \ref{figure10} in Section \ref{section3}. We leave the details to the reader.
\end{proof}

\begin{corollary}
If $w(\widetilde{K})\neq0$, then both $\widetilde{\mathcal{F}}(\widetilde{K})$ and $\widetilde{\mathcal{L}}(\widetilde{K})$ are nontrivial.
\end{corollary}

\begin{proposition}
Suppose $\widetilde{\mathcal{F}}(\widetilde{K})=\sum\limits_{i=1}^na_i\widetilde{K_i}$ $(a_i\neq0)$, then $\widetilde{K_i}\neq\widetilde{K}$ for any $1\leq i\leq n$.
\end{proposition}
\begin{proof}
If $\widetilde{K}$ is a unknot, then we have $\widetilde{\mathcal{F}}(\widetilde{K})=0\neq\widetilde{K}$. If $\widetilde{K}$ is nontrivial, we choose a minimal diagram of it, i.e. one of the diagrams which realizes the minimal flat crossing number. According to the definition of $\widetilde{\mathcal{F}}(\widetilde{K})$, each $\widetilde{K_i}$ has strictly fewer flat crossing points, hence it cannot be equivalent to $\widetilde{K}$.
\end{proof}

This proposition says that the map $\widetilde{\mathcal{F}}$ turns an oriented flat virtual knot $K$ into a linear combination of ``simpler" flat virtual knots. Therefore sometimes we can use known nontrivial flat virtual knots to deduce that some more complex flat virtual knots are nontrivial. In particular, if a flat virtual knot $\widetilde{K}$ has flat crossing number $n+2$, then we have $\widetilde{\mathcal{F}}^n(\widetilde{K})=\widetilde{\mathcal{F}}\circ\cdots\circ\widetilde{\mathcal{F}}(\widetilde{K})=0$. Here $\widetilde{\mathcal{F}}$ should be understood as a linear map $\widetilde{\mathcal{F}}:\mathcal{M}_1\rightarrow\mathcal{M}_1$ by extending linearly.

For a given flat virtual knot $\widetilde{K}$, assume $\widetilde{\mathcal{F}}(\widetilde{K})=\sum\limits_{i=1}^na_i\widetilde{K_i}$ $(a_i\neq0)$, now let us focus on these $\{\widetilde{K_i}\}_{1\leq i\leq n}$ for a moment. Let $K, K'$ be two virtual knot diagrams which represent the same virtual knot, and we use $G(K), G(K')$ to denote the corresponding Gauss diagrams of them. If one chooses a Gauss sub-diagram of $G(K)$, say $G(K'')$, i.e. $G(K'')$ is obtained from $G(K)$ by deleting some chords. In general, it is possible that none of the Gauss sub-diagrams of $G(K')$ represents the same virtual knot as $G(K'')$. In order to see this, consider a nontrivial virtual knot and choose a minimal real crossing number diagram of it, say $K'$. By performing the first Reidemeister move we obtain another diagram $K$ with one more real crossing point. It is obvious that $G(K')$ is a Gauss sub-diagram of $G(K)$. However, no sub-diagram of $G(K')$ represents $K'$ since $K'$ is minimal. Nevertheless, if we consider some special Gauss sub-diagram of $G(K)$, for example $G(P_n(K))$ (see Example \ref{3.18}), then for any other Gauss diagram which represents the same virtual knot as $G(K)$, we can always find a sub-diagram of it such that this sub-diagram represents the same virtual knot as $G(P_n(K))$. In other words, $G(P_n(K))$ is an essential sub-diagram of $G(K)$. For flat virtual knots, some similar projections were discussed by Turaev in \cite{Tur2004}. Come back to $\widetilde{\mathcal{F}}(\widetilde{K})=\sum\limits_{i=1}^na_i\widetilde{K_i}$ $(a_i\neq0)$, it is worth noticing that each $G(\widetilde{K_i})$ can be regarded as an essential twisted sub-diagram of $G(\widetilde{K})$. Here the reason why we use the word ``twisted" is that each $G(\widetilde{K_i})$ can be obtained from $G(\widetilde{K})$ by removing a chord and then performing a half-twist on the left or right semicircle. The choice of left side or right side depends on the writhe of the removed chord. Some more essential and (more complicated) twisted sub-diagrams can be obtained by considering $\widetilde{\mathcal{F}}^n$.

\begin{example}\label{5.8}
Recall the definition of $\widetilde{K_{p,q}}$ in Example \ref{5.3}. Let $p=n\geq3$ and $q=1$, one obtains
\begin{center}
$\widetilde{\mathcal{F}}(\widetilde{K_{n,1}})=
\begin{cases}
U-\sum\limits_{i=1}^k\widetilde{K_{2i-1,1}}-\sum\limits_{i=1}^k\widetilde{K_{1,2i-1}}=-\sum\limits_{i=2}^k\widetilde{K_{2i-1,1}}-\sum\limits_{i=2}^k\widetilde{K_{1,2i-1}}-U& \text{if }n=2k\\
U-\sum\limits_{i=0}^k\widetilde{K_{2i,1}}-\sum\limits_{i=1}^k\widetilde{K_{1,2i}}=-\sum\limits_{i=1}^k\widetilde{K_{2i,1}}-\sum\limits_{i=1}^k\widetilde{K_{1,2i}}& \text{if }n=2k+1
\end{cases}$
\end{center}
It is worth pointing out that in this example we have $\widetilde{\mathcal{F}}^2(\widetilde{K_{n,1}})=0$.
\end{example}

One can easily extend these two invariants from flat virtual knots to (ordered) flat virtual links with a little modifications on the definition of $\mathcal{M}_i$ $(i=1,2)$. Actually, the index of a self-flat crossing point can be defined similarly as above. For a flat crossing point between different components, one can assign a sign to it in a similar way as what we did when we introduce the flat linking number in Example \ref{3.11}. On the other hand, if $w(c)$ is replaced by Ind$(c)$ then one obtains two other flat virtual knot invariants $\widetilde{\mathfrak{F}}(\widetilde{K})=\sum\limits_c\text{Ind}(c)\widetilde{K}_c$ and $\widetilde{\mathfrak{L}}(\widetilde{K})=\sum\limits_c\text{Ind}(c)\widetilde{L}_c$. Even if we always have Ind$(\widetilde{K})=\sum\limits_c\text{Ind}(c)=0$, from Example \ref{5.8} we see that $\widetilde{\mathfrak{F}}(\widetilde{K_{n,1}})\neq0$ if $n\geq3$. The example below suggests that $\widetilde{\mathfrak{L}}(\widetilde{K})$ can be nontrivial when $n=2$.
\begin{figure}[h]
\centering
\includegraphics{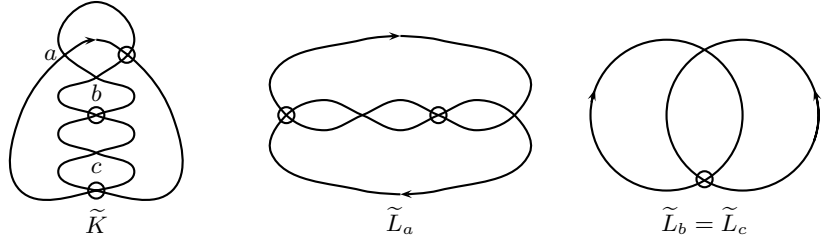}\\
\caption{A nontrivial flat virtual knot $\widetilde{K}$}\label{figure13}
\end{figure}
\begin{example}
Consider the flat virtual knot $\widetilde{K}$ ($=\widetilde{K_{2,1}}$ in the sense of Example \ref{5.3}) in Figure \ref{figure13}. Denote the three flat crossing points by $a, b, c$, respectively. Replacing them with positive crossing points one finds that Ind$(a)=2$, Ind$(b)= \text{Ind}(c)=-1$. Therefore
\begin{center}
$\widetilde{\mathfrak{L}}(\widetilde{K})=2\widetilde{L}_a-\widetilde{L}_b-\widetilde{L}_c=2\widetilde{L}_a-2\widetilde{L}_b$,
\end{center}
since here we have $\widetilde{L}_b=\widetilde{L}_c$. By simply calculating the flat linking number we find that $lk(\widetilde{L}_a)=2$ but $lk(\widetilde{L}_b)=1$. Hence $\widetilde{\mathfrak{L}}(\widetilde{K})\neq0$.
\end{example}

Comparing with $\widetilde{\mathfrak{L}}(\widetilde{K})$, which actually used the notion of chord index for two times, sometimes $\widetilde{\mathcal{L}}(\widetilde{K})$ can tell us some immediate information of the flat virtual knot $\widetilde{K}$. For instance, if $\widetilde{\mathcal{L}}(\widetilde{K})=\sum\limits_{j=1}^nb_j\widetilde{K_j}$, then obviously $\sum\limits_{j=1}^n|b_j|$ provides us a lower bound for the number of flat crossing points of $\widetilde{K}$. Consider the example above, we have $\widetilde{\mathcal{L}}(\widetilde{K})=\widetilde{L}_a-2\widetilde{L}_b$. Thus, the flat crossing number of $\widetilde{K}$ is exactly three. Surely, this also can be concluded from the fact that there is no flat virtual knot with flat crossing number two.

\begin{proposition}
Let $\widetilde{K}$ be a flat virtual knot diagram and $r(\widetilde{K})$ the diagram obtained from $\widetilde{K}$ by reversing the orientation, then $\widetilde{\mathcal{F}}(r(\widetilde{K}))=-r(\widetilde{\mathcal{F}}(\widetilde{K}))$ and $\widetilde{\mathcal{L}}(r(\widetilde{K}))=-r(\widetilde{\mathcal{L}}(\widetilde{K}))$.
\end{proposition}
\begin{proof}
Notice that when the orientation is reversed, the index of each crossing point turns to its opposite. Furthermore, for each flat crossing point $c$ the orientation of $\widetilde{K}_c$ is reversed when the orientation of $\widetilde{K}$ is reversed. For $\widetilde{L}_c$, the orientations of the two components are both changed.
\end{proof}

\section*{Acknowledgement}
The authors are supported by NSFC 11771042 and NSFC 11571038.

\end{document}